\documentclass{amsart}
\usepackage{amsmath,amsthm,amsfonts,amssymb}
\usepackage{mathrsfs}
\usepackage{stmaryrd}
\usepackage{mathtools}
\usepackage{subfiles}
\usepackage{fullpage}
\usepackage{tikz}
\usepackage{float}
\usepackage{cancel}
\usepackage{xcolor}
\usepackage{setspace}
\usepackage{calligra}
\usepackage{MnSymbol}
\RequirePackage{doi}
\usepackage{hyperref}
\usepackage[all]{xy}
\usepackage[capitalize]{cleveref}
\usepackage{thmtools}
\usepackage[shortlabels]{enumitem}
\usepackage[many]{tcolorbox}

\usetikzlibrary{shapes.geometric}
\usetikzlibrary{calc}

\newtheorem{thm}{Theorem}[section]
\newtheorem{lem}[thm]{Lemma}
\newtheorem{cor}[thm]{Corollary}
\newtheorem{prop}[thm]{Proposition}

\newtheorem*{lem*}{Lemma}
\newtheorem*{cor*}{Corollary}
\newtheorem*{prop*}{Proposition}
\newtheorem*{fact*}{Fact}
\newtheorem*{conj*}{Conjecture}
\newtheorem*{clm*}{Claim}

\theoremstyle{definition}
\newtheorem{defn}[thm]{Definition}
\newtheorem{ex}[thm]{Example}

\newtheorem{axiom}[thm]{Axiom}

\newtheorem*{defn*}{Definition}
\newtheorem*{ex*}{Example}
\newtheorem*{xca*}{Exercise}
\newtheorem*{axiom*}{Axiom}

\theoremstyle{remark}
\newtheorem{rmk}[thm]{Remark}

\newtheorem{notation}[thm]{Notation}

\newtheorem*{rmk*}{Remark}
\newtheorem*{setup*}{Setup}
\newtheorem*{cond*}{Condition}
\newtheorem*{cons*}{Construction}
\newtheorem*{obs*}{Observation}
\newtheorem*{ques*}{Question}
\newtheorem*{dis*}{Discussion}

\numberwithin{equation}{section}

\newlist{thmlist}{enumerate}{1}
\setlist[thmlist]{label=(\roman{thmlisti}),
                  ref=\thethm.(\roman{thmlisti}),
                  noitemsep}
                  
\Crefname{listthm}{Theorem}{Theorems}
\Crefname{listlem}{Lemma}{Lemmas}
\Crefname{listprop}{Proposition}{Propositions}
\Crefname{listcor}{Corollary}{Corollaries}
\Crefname{listfact}{Fact}{Facts}
\Crefname{listconj}{Conjecture}{Conjectures}
\Crefname{listclm}{Claim}{Claims}

\addtotheorempostheadhook[thm]{\crefalias{thmlisti}{listthm}}
\addtotheorempostheadhook[lem]{\crefalias{thmlisti}{listlem}}
\addtotheorempostheadhook[prop]{\crefalias{thmlisti}{listprop}}
\addtotheorempostheadhook[cor]{\crefalias{thmlisti}{listcor}}
\addtotheorempostheadhook[fact]{\crefalias{thmlisti}{listfact}}
\addtotheorempostheadhook[clm]{\crefalias{thmlisti}{listclm}}
\addtotheorempostheadhook[conj]{\crefalias{thmlisti}{listconj}}

\DeclareMathOperator{\Ima}{Im}

\DeclareMathOperator{\Hom}{Hom}

\DeclareMathOperator{\Ann}{Ann}

\newcommand{\mbb}[1]{\mathbb{#1}}

\newcommand{\NN}{\mbb{N}}

\newcommand{\msc}[1]{\mathscr{#1}}

\newcommand{\sB}{\msc{B}}

\newcommand{\mf}[1]{\mathfrak{#1}}
\newcommand{\mm}{\mf{m}}

\newcommand{\mrm}[1]{\mathrm{#1}}

\newcommand{\mH}{\mrm{H}}

\newcommand{\mcal}[1]{\mathcal{#1}}

\newcommand{\cC}{\mcal{C}}

\newcommand{\sheafname}[2]{\mathscr{#1}\text{\kern -3pt {\calligra\large  #2}}\,}

\newcommand{\SDer}{\mathscr{D}\text{\kern -2pt {\calligra\large er}}\,}

\newcommand{\seq}[3][1]{#2_{#1},\dots,#2_{#3}}

\DeclarePairedDelimiter\ps{\llbracket}{\rrbracket}
\DeclarePairedDelimiter\set{\{}{\}}

\makeatletter
\renewcommand*{\eqref}[1]{%
  \hyperref[{#1}]{\textup{\tagform@{\ref*{#1}}}}%
}
\makeatother

\definecolor{zhan}{rgb}{0.0, 0.44, 1.0}

\hypersetup{
  colorlinks = true,
  linkcolor = red,
  anchorcolor = blue,
  citecolor = blue,
  filecolor = cyan,
  menucolor = red,
  runcolor = cyan,
  urlcolor = magenta,
}

\newtcolorbox{mybox}[3][beamer]{
    colback=white,
    coltitle=black,
    colframe=#3,
    coltext=zhan,
    title={\textbf{#2}},
    valign=top,
    halign=left,
    before skip=0.5cm,
    after skip=0.5cm,
    attach title to upper,
    after title={:\ },
    #1
}

\usepackage{MnSymbol,xcolor}

\DeclareMathOperator{\cl}{cl}
\DeclareMathOperator{\intr}{i}
\DeclareMathOperator{\epf}{epf}
\DeclareMathOperator{\wepf}{wepf}

\newcommand{\dual}{\smallsmile}

\newcommand{\sop}{system of parameters}
\newcommand{\fg}{finitely-generated}

\newcommand{\CM}{Cohen-Macaulay}
\newcommand{\Z}{\mathbb{Z}}

\title[Rationality for arbitrary closure operations]{Rationality for arbitrary closure operations and the test ideal of full extended plus closure}

\author{Zhan Jiang}
\address{Department of Mathematics, University of Michigan, Ann Arbor, MI 48109}
\email{zoeng@umich.edu}

\author{Rebecca R.G.}
\address{Department of Mathematical Sciences, George Mason University, 
Fairfax, VA 22030}
\email{rrebhuhn@gmu.edu}

\date{\today}

\begin{document}

\maketitle

\begin{abstract}
    We extend the notion of F-rationality to other closure operations, inspired by the work of Smith, Epstein and Schwede, and Ma and Schwede, which describe F-rationality in terms of the canonical module and top local cohomology module. We give conditions for a closure operation cl on a \CM\ complete local ring under which cl-rationality is equivalent to parameter ideals being cl-closed.
    We also demonstrate that full extended plus closure as defined by Heitmann and weak full extended plus closure as defined by the first named author have no big test elements.
\end{abstract}

\section{Introduction}

Rational or F-rational singularities have long been a useful descriptor for rings of equal characteristic \cite{fedderwatanabe,hochsterhunekefrationality}. Given the proliferation of possible closure operations in the mixed characteristic case \cite{MaSc2020,Heitmann2018,jiang,Perez2019}, it is useful to know for which closures cl we can define cl-rational singularities, and what that definition looks like. In this paper we build on two definitions originally used for F-rationality over \CM\ local rings: first, that a ring is F-rational if one or all parameter ideals are tightly closed \cite{fndtc}, and second that a ring is F-rational if the annihilator of the tight closure of 0 in the injective hull of the residue field is the whole ring \cite{SmithThesis}. These ideas were used more recently in \cite{MaSc2020} to define BCM-rationality in terms of the annihilator in the canonical module of their BCM-closure of $(0)$ in $\mH_m^d(R)$. 

Using these ideas, in Section \ref{sec:clrational} we define cl-rationality for closure operations on \CM\ local rings in terms of the canonical module and $\mH_m^d(R)$, and show that under conditions comparable to those used for tight closure, this is equivalent to parameter ideals being cl-closed. We expect these results to be useful for researchers defining new closure operations in any characteristic, who want to check if cl-rationality coincides with rationality in equal characteristic 0, F-rationality in equal characteristic $p$, and/or BCM-rationality in any characteristic. To demonstrate the importance of choosing the right closure, in Section \ref{sec:examples} we show that even F-rational rings may not be $\cl_B$-rational for the module closures coming from some \fg\ \CM\ modules $B$.

The other key result of this paper is that full extended plus closure (epf) as defined by Heitmann \cite{heitmannExtensionsofPlusClosure,HeitmannDirectSummand} does not have big test elements in general (see Section \ref{sec:epf}). This gives a more explicit way to understand how epf closure is ``too big'' to be the right mixed characteristic closure to define singularity types. We further prove that the epf test ideals (big and finitistic) agree with those of wepf as defined by the first named author \cite{jiang}, implying that the latter also lacks big test elements.

\section{Background}

Throughout, $R$ will be a commutative Noetherian ring of Krull dimension $d$. When $(R,m,k)$ is local, $E=E_R(k)$ will denote the injective hull of $k$ and $^\vee$ the Matlis duality operator $\Hom_R(-,E)$.

\begin{defn}
A \textit{closure operation} $\cl$ on a category $\mathcal{M}$ of $R$-modules is a map sending each pair $N \subseteq M$ contained in $\mathcal{M}$ to an $R$-module $N_M^{\cl} \subseteq M$ such that
\begin{enumerate}
    \item $N \subseteq N_M^{\cl}$;
    \item $(N_M^{\cl})_M^{\cl}=N_M^{\cl}$;
    \item and if $N \subseteq N' \subseteq M$ are in $\mathcal{M}$, then $N_M^{\cl} \subseteq (N')_M^{\cl}$.
\end{enumerate}

An \textit{interior operation} $\intr$ on a category $\mathcal{M}$ of $R$-modules is a map sending each $R$-module $M$ to a submodule $\intr(M)$ such that
\begin{enumerate}
    \item $\intr(\intr(M))=\intr(M)$;
    \item and if $L \subseteq M$, then $\intr(L) \subseteq \intr(M)$.
\end{enumerate}
\end{defn}

\begin{notation}
Let $R$ be a ring and $\cl, \cl'$ closure operations on a category $\mathcal{M}$ of $R$-modules. We say that $\cl \leqslant \cl'$ if for all $N \subseteq M$ in $\mathcal{M}$, $N_M^{\cl} \subseteq N_M^{\cl'}$.
\end{notation}

\begin{defn}[\cite{R.G.2016}]
\label{defn:moduleclosure}
Let $B$ be an $R$-module. The \textit{module closure} coming from $B$ is given by
\[L_M^{\cl_B}:=\{x \in M \mid b \otimes x \in \text{im}(B \otimes_R L \to B \otimes_R M) \text{ for all } b \in B\}.\]
\end{defn}

\begin{defn}[{\cite[Definition 2.2]{EpRGVNonresidual}}]
    Let $R$ be a commutative ring and let $\cl$ be a closure operation on a category of $R$-modules $\mathcal{M}$. We say that $\cl$ is \textit{functorial} if whenever $L \subseteq M, N$ are in $\mathcal{M}$ and $f:M \to N$ is a map in $\mathcal{M}$, $f(L_M^{\cl}) \subseteq f(L)_N^{\cl}$.
    
    We say that $\cl$ is \textit{residual} if whenever $L \subseteq N \subseteq M$ in $\mathcal{M}$ such that $M/L,N/L$ are also in $\mathcal{M}$, $N_M^{\cl}=\pi^{-1}((N/L)_{M/L}^{\cl})$ where $\pi:M \to M/L$ is the natural surjection.
\end{defn}

\begin{defn}[{\cite[Definition 3.1]{EpRGVNakayama}}]
Let $\cl$ be a residual closure operation on $R$-modules. The \textit{finitistic version} $\cl_{fg}$ of $\cl$ is given by
\[L_M^{\cl_{fg}}=\bigcup \left\{L_N^{\cl} \mid L \subseteq N \subseteq M, N/L \text{ \fg} \right\}.\]
We say that a residual closure operation $\cl$ is \textit{finitistic} if it is equal to its finitistic version.
\end{defn}

See \cite{EpRGVNonresidual} for a discussion of the finitistic version in the non-residual case.

\begin{defn}[{\cite[Definition 3.9]{R.G.2016}}]
\label{def:strongccversionA}
Let $R$ be local and $\cl$ a closure operation on at least ideals of $R$. We say that $\cl$ satisfies \textit{colon-capturing} if for every partial \sop\ $x_1,\ldots,x_{k+1}$ on $R$,
\[(x_1,\ldots,x_k):x_{k+1} \subseteq (x_1,\ldots,x_k)^{\cl}.\]

We say that $\cl$ satisfies \textit{strong colon-capturing, version A} if for every partial \sop\ $x_1,\ldots,x_k$ on $R$ and $0 \leqslant a < t$, 
\[(x_1^t,x_2,\ldots,x_k):_R x_1^a \subseteq (x_1^{t-a},x_2,\ldots,x_k)^{\cl}. \]

We say that $\cl$ satisfies \textit{strong colon-capturing, version B} if for every partial \sop\ $x_1,\ldots,x_{k+1}$ on $R$,
\[(x_1,x_2,\ldots,x_k)^{\cl} :_R x_{k+1} \subseteq (x_1,x_2,\ldots,x_k)^{\cl}.\]
\end{defn}

\begin{rmk}
\label{rmk:ccimplications}
    Note that strong colon-capturing, version A does not necessarily imply colon-capturing, but strong colon-capturing, version B does imply colon-capturing.
\end{rmk}

\begin{defn}[\cite{EpRG2020}]
\label{defn:dual}
Let $(R,m,k)$ be a complete local ring and $\cl$ a residual closure operation on $R$-modules. We define a dual interior operation $\cl^\dual$ on \fg\ and Artinian $R$-modules, by
\[\cl^\dual(M)=\left(\frac{M^\vee}{0_{M^\vee}^{\cl}} \right)^\vee.\]
Note that for these modules $M \cong M^{\vee \vee}$, so $\cl^\dual(M)$ can be viewed as a submodule of $M$ via this isomorphism.
\end{defn}

\begin{defn}
Let $\cl$ be a closure operation on $R$-modules. The \textit{$\cl$-test ideal} is
\[\tau_{\cl}(R)=\bigcap_{N \subseteq M} N:_R N_M^{\cl},\]
where the intersection is taken over all $R$-modules $N \subseteq M$.

The \textit{finitistic $\cl$-test ideal} is
\[\tau_{\cl}^{fg}(R)=\bigcap_{N \subseteq M \text{ f.g.}} N:_R N_M^{\cl},\]
where the intersection is taken over all \fg\ $R$-modules $N \subseteq M$.
\end{defn}

The following result demonstrates how $\cl^\dual(R)$ can be perceived as the $\cl$-test ideal of $R$.

\begin{lem}[{\cite[Proposition 3.9]{Perez2019},\cite[Theorem 5.5]{EpRG2020}}]
Let $(R,m,k)$ be complete and local.
If $\cl$ is a functorial, residual closure operation on $R$-modules, then $\tau_{\cl}(R)=\Ann_R 0_{E_R(k)}^{\cl}$.

As a result, $\tau_{\cl}(R)=\cl^{\dual}(R)$.
\end{lem}

\begin{defn}[\cite{R.G.2016}]
\label{def:weaklyclregular}
Let $\cl$ be a residual closure operation on $R$-modules. We say that $R$ is weakly-$\cl$-regular if for every ideal $I$ of $R$, $I^{\cl}_R=I$, or equivalently for every pair of \fg\ $R$-modules $N \subseteq M$, $N_M^{\cl}=N$.
\end{defn}

Note that if $\cl$~captures colons, then weakly $\cl$-regular rings are \CM.

\begin{defn}[\cite{fndtc}]
Let $R$ be a local ring of characteristic $p>0$ and let $*$ denote tight closure. We say that $R$ is \textit{F-rational} if $I^*=I$ for any ideal $I$ generated by part of a system of parameters.
\end{defn}

The result below gave the first way of viewing F-rationality in terms of local cohomology:

\begin{thm}[{\cite[Proposition 4.1.4]{SmithThesis}}]
Let $(R,m)$ be a complete local \CM\ ring. Then $R$ is F-rational if and only if $\Ann_R\left(0_{\mH_m^d(R)}^*\right)=R$.
\end{thm}

This perspective has since been used by many others, for example Epstein and Schwede in \cite{EpScTightInterior} and Ma and Schwede in \cite{MaSc2020} for their BCM test ideal:

\begin{defn}[{\cite[Section 5]{MaSc2020}}]
\label{def:maschwedetestideal}
Let $(R,m,k)$ be a complete local ring of dimension $d$ with a normalized dualizing complex $\omega^\bullet$ and canonical module $\omega$. Let $B$ denote a big \CM\ $R$-algebra. We set
\[
0_{\mH^d_{\mm}(R)}^{\sB} = \set{\eta\in \mH^d_{\mm}(R)|\exists B \text{ such that } \eta\in 0^B_{\mH^d_{\mm}(R)}}
\]
and
\[
\tau_{\sB}(\omega_R) = \Ann_{\omega_R} 0^{\sB}_{\mH^d_{\mm}(R)}\subseteq \omega_R
\]
\end{defn}

Ma and Schwede go on to prove that this agrees with the parameter test submodule. This test submodule is given in terms of $\omega_R$ and $\mH_m^d(R)$, rather than in terms of $R$ and $E_R(k)$ as is the tight closure test ideal \cite{hochsterhuneke90,fndtc,testidealssurvey,EpScTightInterior} or the cl-test ideal \cite{Perez2019,EpRG2020,EpRGVNakayama}. 
If $R$ is Gorenstein, the two notions agree since $\omega_R \cong R$ and $E_R(k) \cong \mH_m^d(R)$. However, even when $R$ is not Gorenstein, the former test ideal is still an example of the duality of \cite{EpRG2020}. In Section \ref{sec:clrational} we will explore and apply this duality for residual closure operations.

\section{Full extended plus closure}
\label{sec:epf}

In this section, we prove that full extended plus closure has no big test elements and that its test ideals agree with those of weak full extended plus closure.

\begin{defn}[{\cite{HeitmannDirectSummand,R.G.2016}}]
Let $R$ be a domain. We define the \textit{absolute integral closure} $R^+$ of $R$ to be the integral closure of $R$ in an algebraic closure of its fraction field.

Let $R$ be a mixed characteristic local domain whose residue field has characteristic $p$. Let $N \subseteq M$ be modules over $R$. We define the \textit{full extended plus closure} of $N$ in $M$ to be the set of $u \in M$ such that there exists some nonzero $c \in R$ such that for all $n \in \Z_+$,
\[c^{1/n} \otimes u \in \text{im}\left(R^+ \otimes_R N + R^+ \otimes_R p^nM \to R^+ \otimes_R M\right).\]
\end{defn}

\begin{thm}
\label{thm:epf0}
Let $(R,m,k)$ be a complete local domain of mixed characteristic $p$ with $F$-finite residue field $k$. Then $\tau_{\epf}(R)=0$.
\end{thm}

\begin{proof}
Since $R$ is a domain, the map from $R$ to itself given by multiplication by $p^n$ is injective. Taking its Matlis dual, we see that the map from $E=E_R(k)$ to itself given by multiplication by $p^n$ is surjective. This is preserved by tensoring over $R$ with $R^+$. So for any element $u \in E$, $d^\epsilon \otimes u \in R^+ \otimes p^nE$ for $\epsilon>0$ rational and $d \in R^+$. Hence $0_E^{\epf}=E$, which implies that $\Ann_R(0_E^{\epf})=\Ann_R E=0$. \qedhere
\end{proof}

\begin{defn}[{\cite[Definition 4.1]{jiang}}]
    Let $R$ be a complete local ring whose residue field has characteristic $p>0$. We define the \textit{weak full extended plus closure} of $N$ in $M$ to be 
    \[N_M^{\wepf}:=\bigcap_{n \geqslant 0} (N+p^nM)_M^{\epf}.\]
\end{defn}

\begin{cor}
Let $R$ be a complete local domain of mixed characteristic $p$ with $F$-finite residue field. Then $\tau_{\wepf}(R)=0$.
\end{cor}
\begin{proof}
    Since $0_E^{\wepf} = \bigcap_n (p^nE)^{\epf}_E$ and $E=0_E^{\epf} \subseteq (p^nE)^{\epf}_E\subseteq E$, we have 
    \[0_E^{\wepf} = \bigcap_n (p^nE)_E^{\wepf} = \bigcap_n E=E. \qedhere \]
\end{proof}
 
 However, as a corollary of \cite[Corollary 4.2]{Ma2019a}, we see that full extended plus closure does have finitistic test elements:
 
\begin{cor}
Let $(R, \mm)$ be a complete normal local domain
of residue characteristic $p > 0$ and of dimension $d$. Let $J$ be the defining ideal of the singular
locus of $R$. Then there exists an integer $N$ such that $J^N I^{\epf} \subseteq I$ for all $I\subseteq R$ .
\end{cor}
\begin{proof}
From the proof of \cite[Corollary 4.2]{Ma2019a}, we have
\begin{itemize}
    \item $I^{\epf}\subseteq (I,p^n)B\cap R$ for some fixed perfectoid big Cohen-Macaulay $R^+$-algebra $B$ and every $n$;
    \item There exists some $N$ such that $J^N\subseteq \Ima(\Hom_R(B,R)\to R)$.
\end{itemize}
Then the last paragraph of the proof of \cite[Corollary 4.2]{Ma2019a} works with $\overline{I^h}$ replaced by $I^{\epf}$. Explicitly, for every $r\in J^N$, there exists $\phi\in \Hom_R(B,R)$ such that $\phi(1)=r$. Applying $\phi$ to $I^{\epf}\subseteq (I,p^n)B\cap R$ we get $r I^{\epf}\subseteq (I,p^n)R$ for every $n$. Hence $J^N I^{\epf}\subseteq \bigcap_n (I,p^n)R=I$.
\end{proof}

\begin{thm}
Let $R$ be a complete local domain of mixed characteristic $p>0$. Then the finitistic test ideals for $\epf$ and $\wepf$ are the same. 
\end{thm}

\begin{proof} 
Since $\epf\leqslant \wepf$ \cite[Remark 4.2]{jiang}, any element in $\tau_{\wepf}^{fg}(R)$ will also be in $\tau_{\epf}^{fg}(R)$. For the reverse containment, if $c \in \tau_{\epf}^{fg}(R)$, then $c(I,p^n)^{\epf}\subseteq (I,p^n)$ for all $n>0$. For any $u\in I^{\wepf} = \bigcap_{n\in\NN} (I,p^n)^{\epf}$, we have $cu\in (I,p^n)$ for all $n$, which in particular implies that $cu\in I$. Hence, $c\in \tau_{\wepf}^{fg}(R)$. 
\end{proof}

\begin{rmk}
 We do not know if or where these two closure operations agree, despite the above result. If $\epf=\wepf$ in general, then this implies that $\epf$ is a Dietz closure \cite{dietz}. If they do not always agree, then they provide an interesting example of distinct closures with the same finitistic and big test ideals. 
\end{rmk}

\begin{rmk}
    The results above indicate that the finitistic and big test ideals do not coincide for either epf or wepf. This gives a partial answer to \cite[Question 3.7]{Perez2019}, which asks which closure operations have the property that the big and finitistic test ideals coincide. For finitistic closure operations like Frobenius and plus closure, the big and finitistic test ideals are known to coincide, and it is conjectured that they coincide for tight closure.
\end{rmk}

\section{Canonical modules and cl-rationality}
\label{sec:clrational}

In this section we define cl-rational singularities for residual, functorial closure operations cl over \CM\ local rings. Under hypotheses comparable to those used for F-rationality, we show that cl-rational rings are those rings whose parameter ideals are cl-closed.

We begin by setting up the necessary results from closure-interior duality.

\begin{defn}
    Let $(R,\mm,k)$ be a complete local ring. Write $\vee$ for the Matlis duality operator $\Hom_R(-,E)$, where $E$ is the injective hull of $k$. If $M$ is a \fg\ or Artinian $R$-module and $N \subseteq M^\vee$, we define
    \[\Ann_M N:=\{f \in \Hom_R(M^\vee,E) \mid f(N)=0\}.\]
    Note that by the isomorphism $M^{\vee \vee} \cong M$, we may view this as a submodule of $M$.
\end{defn}

\begin{prop}
Let $(R,\mm,k)$ be a complete local ring and $M$ a \fg\ or Artinian $R$-module. Then for any module $M$ that is \fg\ or Artinian and any submodule $N\subseteq M^{\vee}$ we have
\[
\left(\frac{ M^{\vee} }{ N} \right)^{\vee} = \Ann_M N
\]
\end{prop}

\begin{proof}
\begin{align*}
    f\in \left(\frac{ M^{\vee} }{ N} \right)^{\vee} & \iff f\in \Hom_R\left(\frac{ M^{\vee} }{ N},E \right)\\
    &\iff  f\in \Hom_R\left( M^{\vee},E \right)\text{ and $f$ kills }N\\
    &\iff  f\in (M^{\vee})^{\vee}=M\text{ and $f$ kills }N\\
    &\iff  f\in \Ann_M (N). \qedhere
\end{align*}
\end{proof}

\begin{cor}
\label{cor:dualclosure}
Let $R$ be a complete local ring and $\cl$ a closure operation on $R$-modules. Then for $M$ a \fg\ or Artinian $R$-module,
\[
\left(\frac{ M^{\vee} }{ 0_{M^{\vee}}^{\cl}} \right)^{\vee} = \Ann_M (0_{M^{\vee}}^{\cl}).
\]
In particular, 
\[
\left(\frac{ \mH_m^d(R) }{ 0_{\mH_m^d(R)}^{\cl}} \right)^{\vee} = \Ann_{\omega} (0_{\mH_m^d(R)}^{\cl}),
\]
where $\omega$ is the canonical module of $R$ and $R$ has Krull dimension $d$.
\end{cor}

\begin{proof}
The first part follows by setting $N=0_{M^\vee}^{\cl}$.
\end{proof}

\begin{cor}
\label{cor:whenmoduleis0}
Let $R$ be a complete local ring and $M$ a \fg\ or Artinian $R$-module. Let $N \subseteq M^\vee$. Then $N=0$ if and only if $\Ann_M N=M$.
\end{cor}

\begin{proof}
The forward direction is immediate. For the reverse direction, note that since $E$ is injective, any map $N \to E$ extends to a map $M^\vee \to E$, i.e., an element of $M^{\vee \vee} \cong M$. By hypothesis, this map kills $N$. Hence, $\Hom_R(N,E)=0$, which implies that $N=0$.
\end{proof}

\begin{cor}
\label{cor:omegainterior}
Let $R$ be a complete local ring of Krull dimension $d$ with canonical module $\omega$ and $\cl$ a residual closure operation on $R$-modules. Then
\begin{align*}
\cl^\dual(\omega) 
=\Ann_{\omega} 0_{\mH_m^d(R)}^{\cl}. 
\end{align*}
\end{cor}

\begin{proof}
This follows from Corollary \ref{cor:dualclosure} and Definition \ref{defn:dual}.
\end{proof}

As a result, we see that $\cl^{\dual}(\omega)$ gives us a test submodule generalizing Definition \ref{def:maschwedetestideal} in the same way that the dual of $\cl^{\dual}(R)$ gives us the traditional test ideal as in \cite{EpScTightInterior,EpRG2020}.

\begin{defn}
    We define $\tau_{\cl}(\omega):=\Ann_\omega 0_{\mH_m^d(R)}^{\cl}$.
\end{defn}

We apply this to define cl-rationality:

\begin{defn}
\label{def:clrational}
Let $R$ be a Cohen-Macaulay complete local ring with canonical module $\omega$, and $\cl$ a closure operation on $R$-modules. We say that $R$ is $\cl$-rational if $\tau_{\cl}(\omega)=\omega$.
\end{defn}

\begin{lem}
\label{lem:0closureis0}
Let $R$ be a \CM\ complete local ring of Krull dimension $d$ with canonical module $\omega$ and $\cl$ a residual closure operation on $R$-modules. Then $R$ is $\cl$-rational if and only if $0_{\mH_{m}^d(R)}^{\cl}=0$.
\end{lem}

\begin{proof}
    This follows from Definition \ref{def:clrational} and Corollary \ref{cor:whenmoduleis0}.
\end{proof}

\begin{prop}
Let $R$ be a \CM\ complete local ring of Krull dimension $d$ with canonical module $\omega$. Let $\cl \leqslant \cl'$ be closure operations on $R$-modules. If $R$ is $\cl'$-rational, then $R$ is $\cl$-rational.
\end{prop}

\begin{proof}
If $\cl \leqslant \cl'$, then $\tau_{\cl'}(\omega) \subseteq \tau_{\cl}(\omega)$. So if $\tau_{\cl'}(\omega)=\omega$, then $\tau_{\cl}(\omega)=\omega$.
\end{proof}

\begin{axiom}\label{axiom:finite}
Let $(R,m)$ be a \CM\ local ring of Krull dimension $d$ with canonical module $\omega$, and $\cl$ a residual closure operation on $R$-modules. We say that $\cl$ satisfies the  \textit{injective finiteness condition} if $0_{\mH_m^d(R)}^{\cl}=0_{\mH_m^d(R)}^{\cl_{fg}}$. In particular, finitistic closure operations satisfy this condition.
\end{axiom}

The following result indicates that our definition of $\cl$-rationality often coincides with the more ideal-theoretic definition, as in the case of F-rationality \cite{smith1994tight,fndtc}.

\begin{thm}
\label{thm:clrat}
Let $R$ be a \CM\ complete local ring with canonical module $\omega$, and $\cl$ a residual, functorial closure operation on $R$-modules. If $R$ is $\cl$-rational, then every ideal generated by a system of parameters is $\cl$-closed. If in addition $\cl$ satisfies Axiom \ref{axiom:finite}, then the converse holds.
\end{thm}

\begin{proof}
By Lemma \ref{lem:0closureis0}, $R$ is $\cl$-rational if and only if $0_{\mH_m^d(R)}^{\cl}=0$. 

First we prove the converse. Since $R$ is \CM, for any system of parameters $(x_1,\ldots,x_d)$, 
\[\mH_m^d(R)=\varinjlim_{t} R/(x_1^t,\ldots,x_d^t)R.\]
By our hypotheses, 
\[0_{\mH_d^m(R)}^{\cl}=0_{\mH_d^m(R)}^{\cl_{fg}}=\sum_{G} 0_G^{\cl}, \]
where $G$ ranges over the \fg\ submodules of $\mH_d^m(R)$. These include the $R/(x_1^t,\ldots,x_d^t)$, so \[\sum_t 0_{R/(x_1^t,\ldots,x_d^t)}^{\cl} \subseteq 0_{\mH_d^m(R)}^{\cl_{fg}}.\]
Since each $G$ must be contained in $R/(x_1^t,\ldots,x_d^t)$ for some $t$ and $\cl$ is functorial,
\[0_G^{\cl} \subseteq 0_{R/(x_1^t,\ldots,x_d^t)}^{\cl}\]
for some $t$, so 
\[0_{\mH_d^m(R)}^{\cl_{fg}} \subseteq \sum_t 0_{R/(x_1^t,\ldots,x_d^t)}^{\cl}.\]
Hence 
\[0_{\mH_d^m(R)}^{\cl}=\sum_t 0_{R/(x_1^t,\ldots,x_d^t)}^{\cl}.\]

If every ideal generated by a system of parameters is $\cl$-closed, then since $\cl$ is residual, $0$ is $\cl$-closed in every $R/(x_1^t,\ldots,x_d^t)$. Hence 
$0_{\mH_m^d(R)}=0$.

Conversely, assume $0_{\mH_m^d(R)}=0$, and let $x_1,\ldots,x_d$ be a \sop\ for $R$. Suppose there is some $t$ such that $0_{R/(x_1^t,\ldots,x_d^t)}^{\cl} \ne 0$. Say $u \in 0_{R/(x_1^t,\ldots,x_d^t)}^{\cl}$ is nonzero. Since $\cl$ is functorial, the image of $u$ in $\mH_m^d(R)$ must be in $0_{\mH_m^d(R)}^{\cl}=0.$
Thus there is some $s \geqslant 0$ such that $(x_1 \cdots x_d)^su \in (x_1^{t+s},\ldots,x_d^{t+s})$. Since $R$ is \CM, $x_1,\ldots,x_d$ is a regular sequence, and so $u \in (x_1^t,\ldots,x_d^t)$, giving a contradiction. Hence $0$ is $\cl$-closed in $R/(x_1^t,\ldots,x_d^t)$ for every system of parameters $x_1,\ldots,x_d$. Since $\cl$ is residual, this implies every ideal generated by a system of parameters is $\cl$-closed.
\end{proof}

\begin{rmk}
Note that while tight closure is not known to be finitistic in general, Smith proved that $0_{\mH_m^d(R)}^*=0_{\mH_m^d(R)}^{*_{fg}}$ as long as $R$ is a reduced, equidimensional excellent local ring \cite[text after 3.3]{smith1994tight}, so Theorem \ref{thm:clrat} holds for tight closure. Our proof above is modelled on the tight closure proof in \cite{SmithThesis}.
\end{rmk}

\begin{ex}
We give an example to illustrate the need for Axiom \ref{axiom:finite} in the statement of Theorem \ref{thm:clrat}. Consider epf closure. Let $R$ be a complete regular local ring of mixed characteristic. By Theorem \ref{thm:epf0}, $0_E^{\epf}=E$. Since $R$ is regular, $E=\mH_m^d(R)$. However, by \cite[Theorem 3.19]{Heitmann2018}, every ideal of $R$ is epf-closed in $R$. Since epf-closure is residual \cite[Lemma 3.1 and Proposition 7.2]{R.G.2016},  $0_{R/(x_1^t,\ldots,x_n^t)R}^{\epf}=0$ for all $t \geqslant 0$.
\end{ex}

The next few results give weaker conditions that are equivalent to all parameter ideals being $\cl$-closed, paralleling those for tight closure.

\begin{cor}
\label{cor:rattosop}
If every ideal generated by a \sop\ is $\cl$-closed and $I$ is generated by part of a system of parameters, then $I^{\cl}_R=I$.

In particular, this holds if $R$ is $\cl$-rational.
\end{cor}

\begin{proof}
We use the method of \cite[Page 126]{fndtc}: let $I=(x_1,\ldots,x_k)$, where $x_1,\ldots,x_d$ is a system of parameters for $R$. Then for every $t$, $J_t=(x_1,\ldots,x_k,x_{k+1}^t,\ldots,x_d^t)$ is $\cl$-closed and $I^{\cl} \subseteq J_t^{\cl}=J_t$. Hence $I^{\cl} \subseteq \bigcap_t J_t=I$.

The final claim follows from Theorem \ref{thm:clrat}.
\end{proof}

\begin{lem}
\label{lem:ccimpliescm}
Let $R$ be a local ring and $\cl$ a closure operation on $R$-modules satisfying colon-capturing such that every ideal $I$ generated by part of a system of parameters is $\cl$-closed. Then $R$ is \CM.
\end{lem}

\begin{proof}
Let $I=(x_1,\ldots,x_d)$ where $x_1,\ldots,x_d$ form a system of parameters on $R$. Then for each $1 \leqslant i \leqslant d$, 
\[(x_1,\ldots,x_{i-1}):x_i \subseteq (x_1,\ldots,x_{i-1})^{\cl}=(x_1,\ldots,x_{i-1}).\]
Hence, $x_1,\ldots,x_d$ is a regular sequence, so $R$ is \CM.
\end{proof}

To prove that a ring is $F$-rational, it suffices to prove that one system of parameters is tightly closed under some conditions (reduced, local, has a test element) \cite[Page 128]{fndtc}. In the next couple of results, we prove comparable results for other closure operations.

\begin{prop}\label{prop-cl-closed-for-all}
Let $R$ be a complete local \CM\ ring. Suppose $I=(x_1,\ldots,x_d)$ is generated by a system of parameters. If $\cl$ is a functorial, residual closure operation and $I_t=(x_1^t,\ldots,x_d^t)$ is $\cl$-closed for all $t>0$, then every ideal generated by a \sop\ is $\cl$-closed.

If in addition $\cl$ satisfies \cref{axiom:finite}, then $R$ is $\cl$-rational.
\end{prop}

\begin{proof}
Let $y_1,\ldots,y_d$ be a system of parameters on $R$. Then for $t \gg 0$, $(x_1^t,\ldots,x_d^t) \subseteq (y_1,\ldots,y_d)$. By \cite[Page 122]{fndtc}, there is an injective $R$-module map
\[f:A_1=R/(y_1,\ldots,y_d)R \hookrightarrow R/(x_1^t,\ldots,x_d^t)R=A_2.\]
Since $\cl$ is functorial, $f\left(0_{A_1}^{\cl}\right) \subseteq 0_{A_2}^{\cl}=0$. Since $f$ is injective, $0_{A_1}^{\cl}=0$. Since $\cl$ is residual, $(y_1,\ldots,y_d)_R^{\cl}=(y_1,\ldots,y_d)$.

The final statement follows from Theorem \ref{thm:clrat}.
\end{proof}

\begin{thm}
\label{thm:onesopenough}
Let $R$ be a complete local \CM\ ring. Suppose that $\cl$ is a functorial and residual closure operation satisfying colon-capturing and strong colon capturing Version A (Definition \ref{def:strongccversionA}). Then if $I=(\seq{x}{d})$ generated by one \sop\ is $\cl$-closed, then every ideal generated by a \sop\ is cl-closed.

If $\cl$ also satisfies Axiom \ref{axiom:finite}, $R$ is $\cl$-rational.
\end{thm}

\begin{proof}
Write $I_t=(\seq{x^t}{d})R$. We aim to show that $I_t$ is $\cl$-closed for all $t\geqslant 1$. If not, by \cite[Page 126]{fndtc}, we can assume that for some $t$, there exists $u=x_1^{t-1}\dots x_d^{t-1}z \in I_t^{\cl}-I_t$ where $z$ represents an element of the socle in $R/I$.

Since $\cl$ satisfies strong colon-capturing, version A, $(\seq{x^t}{d})^{\cl}:(x_1\dots x_d)^{t-1}\subseteq (\seq{x}{d})^{\cl}$. Hence $z$ is in $(\seq{x}{d})^{\cl}=(\seq{x}{d})$, a contradiction.

Now by Proposition \ref{prop-cl-closed-for-all}, every ideal generated by a \sop\ is $\cl$-closed.

The final statement follows by \cref{thm:clrat}.
\end{proof}

\begin{thm}
\label{thm:dontneedcm}
    Let $R$ be a local ring, $x_1,\ldots,x_d$ a \sop\ on $R$, and $\cl$ a closure operation on at least ideals of $R$ satisfying strong colon-capturing, version B. Let $I_k$ denote the ideal $(x_1,\ldots,x_k)$. If $I_d^{\cl}=I_d$, then $I_k^{\cl}=I_k$ for each $0 \le k \le d$ and $x_1,\ldots,x_d$ is a regular sequence on $R$.

    Consequently, $R$ is \CM, so if $R$ is complete and $\cl$ is functorial and residual, then every ideal generated by a \sop\ is $\cl$-closed. If in addition $\cl$ satisfies Axiom \ref{axiom:finite}, then $R$ is $\cl$-rational.
\end{thm}

\begin{proof}
    The first paragraph follows as in the proof of the first Theorem on page 128 of \cite{fndtc}, substituting strong colon-capturing, version B for the use of the Theorem cited in the proof.

    The second paragraph follows from Proposition \ref{prop-cl-closed-for-all}.
\end{proof}

\begin{cor}\label{cor:module-closure}
Let $B$ be a big \CM\ $R$-module. Then $R$ is $\cl_B$-rational if and only if one ideal generated by a system of parameters is $\cl_B$-closed.
\end{cor}

\begin{proof}
By \cite{R.G.2016}, $\cl_B$ is a functorial, residual closure operation satisfying \cref{axiom:finite} and strong colon-capturing, versions A and B. By Theorem \ref{thm:dontneedcm}, if the ideal generated by one system of parameters is $\cl_B$-closed, then $R$ is $\cl_B$-rational.

The other direction follows from Corollary \ref{cor:rattosop}.
\end{proof}

\begin{cor}
\label{cor:weaklyimpliesrational}
Let $R$ be a local ring and $\cl$ a functorial, residual closure operation on $R$-modules. If $R$ is weakly $\cl$-regular, then every ideal generated by a \sop\ is $\cl$-closed. Consequently, if $R$ is complete and $\cl$ also satisfies Axiom \ref{axiom:finite}, then $R$ is $\cl$-rational.
\end{cor}

\begin{proof}
This follows from Theorem \ref{thm:clrat} and the definition of weakly $\cl$-regular (Definition \ref{def:weaklyclregular}).
\end{proof}

However, when $R$ is Gorenstein, the two conditions are equivalent, just as they are for tight closure \cite[Page 129]{fndtc}. The result below is a variant of \cite[Corollary 3.6]{EpRGVNakayama}, with a different style of proof.

\begin{thm}
Let $(R,\mm,K)$ be a complete local Gorenstein ring, and $\cl$ a residual, functorial closure operation on $R$-modules satisfying Axiom \ref{axiom:finite}. Then TFAE:
\begin{enumerate}
    \item $R$ is weakly $\cl$-regular, i.e., every submodule of finitely generated modules are $\cl$-closed.
    \item $R$ is $\cl$-rational.
\end{enumerate}
\end{thm}

\begin{proof}
(1) $\Rightarrow$ (2) follows from Corollary \ref{cor:weaklyimpliesrational}. We only need to show the converse.
Assume that $R$ is $\cl$-rational, and let $N\subseteq M$ be finitely generated $R$-modules. If $N$ is not $\cl$-closed, choose $u\in N^{\cl}-N$. We may replace $N$ by $N'$ where $N\subseteq N'\subseteq M$ such that $N'$ is maximal with respect to the property of not containing $u$. We still have $u \in (N')^{\cl}-N'$. The maximality of $N'$ implies that $u$ is in every nonzero submodule of $M/N'$. We replace $M,N'$ by $M/N'$ and $0$ respectively, and $u$ by its image in $M/N'$.  Since
$\cl$ is residual, $0\neq u\in 0^{\cl}_M$.

Then $M$ has finite length and $u$ spans the socle of $M$ by \cite[Lemma on page 56]{fndtc}. Let $\seq{x}{d}$ be a system of parameters for $R$, so that  $A=R/(\seq{x^t}{d})R$ is an Artinian Gorenstein local ring. Then $Ru \cong K \hookrightarrow A$, where the second map sends $y \mapsto y(x_1 \cdots x_d)^{t-1}$. Since $\cl$ is functorial, $u \in 0_A^{\cl}$. But since $\cl$ is residual and $R$ is $\cl$-rational, $0_A^{\cl}=0$, giving us a contradiction.
\end{proof}

We end by describing conditions on a closure operation $\cl$ that imply that a $\cl$-rational ring is a normal domain, something which is known to be true for tight closure in characteristic $p>0$ \cite{fndtc}.

\begin{rmk}\label{threeconditions}
    To imitate the proof that an $F$-rational ring is a normal domain from \cite{fndtc}, one needs the following hypotheses:
    \begin{enumerate}
        \item The closure operation agrees with integral closure on the $0$ ideal and ideals generated by a single nonzerodivisor.
        \item If the ring is not a domain, we need $u \in (f)^{\cl}_R$ if and only if $u \in (\bar{f})_{R/P}^{\cl}$ for every minimal prime $P$.
        \item For $R$-modules $S$ and $T$, and $s \in S$, $(Rs)^{\cl}_S \times T=(Rs \times T)^{\cl}_{S \times T}$ and the same in the other coordinate.
    \end{enumerate}
     
    The next part of our discussion will primarily focus on finding ways to ensure the first condition holds. The second condition can be circumvented by requiring the ring to be a domain (but continuing to use these conditions to prove the ring is normal), and the third is likely to be more straightforward to prove for many common closure operations.

    The part of the first condition involving $(0)$ is to get the ring to be reduced if and only if $(0)^{\cl}_R=(0)$, which can be circumvented in the case of a domain. The other part of the first condition is trickier.

    Paralleling the original result for tight closure, one axiom that would be sufficient to guarantee the rest of the first condition is Axiom 3.7 of \cite{murayama2023uniform}, which is a Brian\c{c}on-Skoda type condition. This would imply in particular for a principal ideal $I$ that $\overline{I} \subseteq I^{\cl}$. As Murayama points out in \cite[Table 1]{murayama2023uniform}, most known Dietz closures satisfy this axiom.
\end{rmk}

Another way to achieve the first condition is by a result of the first named author on plus closure. 
We begin by noting a useful fact about plus closure, including a proof for completeness:

\begin{prop}
\label{pr:plusintegralequal}
Let $R$ be a domain. Let $I$ be a principal ideal of $R$. Then $\overline{I}= I^+$. 
\end{prop}
\begin{proof}
We first show that $I^+\subseteq \overline{I}$.
Suppose that $r\in IR^+$, then $r\in IS$ for some module-finite extension $S$ of $R$. But then $r\in IS\cap R\subseteq \overline{I}$.

For the converse inclusion, note that if $r\in \overline{I}$, then $r\in \overline{IR^+}$. By the Lemma stated below, $r\in IR^+$. Since $R^+$ is normal, we conclude that $r\in I^+$.
\end{proof}

\begin{lem}[{\cite[Page 13]{615notes}}]
Let $I$ be an ideal of the ring $R$ , $r\in I$, and $h:R\to S$ a homomorphism to a normal domain $S$ such that $IS$ is principal. Then $h(r)\in IS$.
\end{lem}

\begin{axiom} [{\cite[Axiom 4.2.2]{Jiangthesis}}]
\label{gcl-axiom-Persistence}
Suppose that $\cC$ is a collections of rings and homomorphisms among them. Let $\cl$ be a closure defined on modules over each ring in $\cC$. If for any homomorphism $R\to S $ in $\cC$ and any $R$-module $M$ and a submodule $N$, we have
  \begin{equation*}
    \Ima(S\otimes_R N_M^{\cl}\rightarrow S\otimes_RM) \subseteq \left( \Ima (S\otimes_RN\rightarrow S\otimes_RM) \right)^{\cl}_{S\otimes_R M},
  \end{equation*}
  then we say that $\cl$ is a \textit{persistent closure with respect to $\cC$}.
\end{axiom}

We point out that the tight closure satisfies Axiom \ref{gcl-axiom-Persistence} and Axiom 3.7 in \cite{murayama2023uniform} under mild hypothesis on the rings 
\cite[Theorem 6.23 and 6.24]{hochsterhunekefrationality}.

\begin{notation}
Let $\Lambda$ be a Noetherian domain and let $\mathcal{C}$ be some collections of $\Lambda$-algebras containing all of the finitely generated $\Lambda$-algebras such that $\cl$ is defined on modules over every ring in $\mathcal{C}$.
\end{notation}

In the results below, we will assume that $\cl$ is persistent with respect to this collection $\mathcal{C}$ in the sense of \cref{gcl-axiom-Persistence}.

\begin{cor}
    Let $R$ be a local domain in $\mathcal{C}$. Suppose that $\cl$ is a persistent (\cref{gcl-axiom-Persistence}), colon-capturing closure operation defined on $\mathcal{C}$ mentioned above. For any $R$ in $\mathcal{C}$,
    and any principal ideal $I\subseteq R$, $\overline{I}\subseteq I^{\cl}$.
\end{cor}

\begin{proof}
         By 
\cite[Theorem 4.2.6]{Jiangthesis}, $I^+ \subseteq I^{\cl}$ for any proper ideal $I$ of $R$, and by Proposition \ref{pr:plusintegralequal}, $I^+=\overline{I}$.
\end{proof}

The corollary above
gives us the following. 

\begin{cor}
    Suppose that $\cl$ is a persistent (\cref{gcl-axiom-Persistence}), colon-capturing closure operation defined on $\mathcal{C}$ mentioned above. For any $R$ in $\mathcal{C}$ that is $\cl$-rational, $\overline{I}= I$ for any principal ideal $I\subseteq R$. Consequently, condition (1) of Remark \ref{threeconditions} holds for $\cl$.
\end{cor}

\section{Examples}
\label{sec:examples}

In this section, we give examples of rings that are and are not $\cl$-rational for certain closure operations $\cl$. In particular, we find rings of finite \CM\ type that are not $\cl_M$-rational for some of their MCM modules $M$, despite having rational singularities \cite{LeuschkeWiegand} and hence F-rational singularities for large enough characteristics $p>0$ \cite{hara}. This makes a case that MCM module closures do not quite capture the right notion of singularity to align with existing classes of singularities.

\begin{prop}
Let $R=k\ps{x^d,x^{d-1}y,\ldots,xy^{d-1},y^d}$, where $k$ is a field. Let $S=k\ps{x,y}$ so that $R$ is a (complete) Veronese subring of $S$. 
For any integer $0 \leqslant i \leqslant d-1$, let $M_i$ be the MCM $R$-module generated by the degree $i$ monomials in $S$, i.e., $M_i=x^iR+x^{i-1}yR+\ldots+y^iR$. 
Then
\begin{enumerate}
    \item $R$ is $\cl_{M_i}$-rational for $0\leqslant i\leqslant d-2$.
    \item $R$ is not $\cl_{M_{d-1}}$-rational.
\end{enumerate}
\end{prop}

\begin{proof}
By \cite[Corollary 6.4]{LeuschkeWiegand}, the $M_i$ are indecomposable MCM modules over $R$ since they are isomorphic to the direct summands of $S$ as an $R$-module. Then by \cref{cor:module-closure}, for each $0 \leqslant i \leqslant d-1$, $R$ is $\cl_{M_i}$-rational if and only if some parameter ideal $I$ is $\cl_{M_i}$-closed.

Let $I=(x^d,y^d)R$. Then $I$ is a parameter ideal of $R$.
We use the following isomorphism
\[M_i = x^iR+x^{i-1}yR+\ldots+y^iR \cong (x^d,x^{d-1}y,\ldots,x^{d-i}y^i)R\]
to view $M_i$ as an ideal $I_i$ of $R$ for $0 \leqslant i \leqslant d-1$. 
Then $I$ is $\cl_{M_i}$-closed if and only if $(II_i):_R I_i= I$. 
First we show that $I$ is not $\cl_{M_{d-1}}$-closed. Note that 
\[(x^d,y^d)\cdot (x^d,x^{d-1}y,\ldots,xy^{d-1})=(x^{2d},x^{2d-1}y,\ldots,x^{d+1}y^{d-1},x^dy^d,\ldots,xy^{2d-1}),\]
which is equal to
\[(x^d,x^{d-1}y,\ldots,xy^{d-1},y^d) \cdot (x^d,x^{d-1}y,\ldots,xy^{d-1}).\]
So $(x^{d-1}y,\ldots,xy^{d-1})$ is contained in $I^{\cl_{M_{d-1}}}$, which implies that $I$ is not $\cl_{M_{d-1}}$-closed.

The fact that $I$ is $\cl_{M_i}$-closed for $0 \leqslant i \leqslant d-2$ follows from the next lemma.
\end{proof}

\begin{lem}
Using the previous notation, if $0 \leqslant i \leqslant d-2$ and $J\cdot I_i \subseteq I\cdot I_i$, then $J \subseteq I$.
\end{lem}

\begin{proof}
Since $I, I_i,$ and $I \cdot I_i$ are monomial ideals, we can assume that $J$ is a monomial ideal. 
Let $x^{a}y^{b}$ be a generator of $J$. Then $a+b$ is divisible by $d$. If $a+b > d$, then $a+b\geqslant 2d$. Hence, either $a\geqslant d$ or $b\geqslant d$, which implies that $x^ay^b \in I$, so we may assume that $a+b=d$.

Let $x^ay^b\in J$ be a monomial not in $I$.  
Then $a\neq 0, b \neq 0$. Let $j=\min(d-b-1, i)$. Then $x^ay^b \cdot x^{d-j}y^j=x^{a+d-j}y^{b+j}\in I\cdot I_i$. Since $j,b+j \leqslant d-1$, we conclude that 
\begin{align*}
x^{a+d-j}y^{b+j}&\in x^d I_i \\
\Leftrightarrow a+d-j &\geqslant 2d-i \\
\Leftrightarrow  2d-b-j &\geqslant 2d-i \\
\Leftrightarrow  i &\geqslant b+j. \\
\end{align*}
Looking at the choice of $j$, there are two possibilities:
\begin{enumerate}
\item $j=i$, which contradicts $i \geqslant b+j>j$.
\item $j=d-b-1 \leqslant i \Leftrightarrow j+b = d-1$, which contradicts $d-2\geqslant i \geqslant b+j$.
\end{enumerate}
Either way we have $x^ay^b\cdot I_i\neq I\cdot I_i $, which is a contradiction. Hence $J \subseteq I$.
\end{proof}

\begin{rmk}
The canonical module of the ring $R$ above is $\omega=M_{d-2}$ (\cite[Corollary 3.1.3]{10.2969/jmsj/03020179}). The above proof shows that $R$ is $\cl_{\omega}$-rational, as expected given that $R$ is F-rational when $k$ has characteristic $p>0$ \cite[5.1]{singh2000veronese}.
\end{rmk}

\begin{prop}
    Let $R=k\ps{x,y}/(x^2y)$. $R$ is neither $\mathrm{cl}_{M_1}$-rational, nor $\mathrm{cl}_{M_2}$-rational, where $M_1=R/(y)$ and $M_2=R/(x^2)$ represent the two isomorphism classes of indecomposable MCM $R$-modules.
\end{prop}

\begin{proof}
    The classification of the indecomposable MCM $R$-modules comes from \cite[Example 14.23]{LeuschkeWiegand}. Since $I=(x+y)$ is a parameter ideal, we only need to show that $I^{\mathrm{cl}}$ is strictly larger than $I$ for $\mathrm{cl}=\mathrm{cl}_{M_1}$ and $\mathrm{cl}_{M_2}$.

    By definition of $\mathrm{cl}_{M_1}$, it is clear that $y\in I^{\mathrm{cl}_{M_1}}$. We only need to show that $y\nin I$, but this is immediate since $I\neq \mathfrak{m}=(x,y)R$.

    The proof is similar for $M_2$: $x^2\in I^{\cl_{M_2}}$ but $x^2\nin I$.
\end{proof}

\begin{prop} 
    Let $R=k\ps{x,y}/(y^2)$ where $k$ is a perfect field. Let
    \[
    \mathcal M = \set{I_n:n\in \mathbb N \cup \set{\infty}},
    \]
    where $I_n=(x^n,y)$.
    Then $R$ is not $I_n$-rational for all $n$.
\end{prop}

\begin{proof}
The set $\mathcal M$ gives a representative for each isomorphism class of indecomposable MCM $R$-modules \cite[Example 6.5]{yoshino}.
As a result, if $R$ is $\cl_{I_n}$-rational then $I_nI:_R I_n=I$ for the parameter ideal $I=(x)$. But $y \cdot I_n= y(x^n, y) = (x^ny) \subseteq (x^{n+1},xy)R)=I_nI$, so $y\in I_nI:_R I$ even though $y\nin I$. So $R$ is not $\cl_{I_n}$-rational.
\end{proof}

\begin{thm}
    Let $R=k\ps{x,y}/(x^n+y^2)$ where $k$ is an algebraically closed field and $n$ is an odd positive integer. Consider the set $\set{R, (x,y), (x^2,y), \ldots,(x^{\frac{n-1}{2}},y)}$ of ideals of $R$. If $n>1$, then $R$ is not $\cl_{(x^i, y)R}$-rational for $i\geqslant 1$.
\end{thm}
\begin{proof}
    By \cite[Proposition 5.11]{yoshino}, this is a set of representatives of the isomorphism classes of indecomposable MCM modules of $R$.
    Note that $I=(x)R$ is a parameter ideal since $y^2=-x^n\in I \Rightarrow \mathfrak m = \sqrt{I}$. Assume $n>1$. We only need to show that $I_iI:_R I_i \neq I$ where $I_i=(x^i,y)R$ for $i=1,...,\frac{n-1}{2}$.

    Note that $y\cdot I_i = (x^iy, y^2)=(x^iy, x^n)\subseteq (x^{i+1}, xy) = I_i\cdot I$, but $y\nin I$.
\end{proof}

\begin{prop}
    Let $R$ be a 2-dimensional ADE singularity, so that $R$ is a ring of the form $k\ps{x,y,z}/(z^2+g(x,y))$ where $g(x,y)\in k\ps{x,y}$. Let $M$ be a MCM $R$-module, that is, a cokernel of a matrix of the form $zI_n-\varphi$ where $\varphi$ is an $n \times n$ square matrix with entries in $ (x,y)k\ps{x,y}$ and $I_n$ is the $n \times n$ identity matrix. Then $R$ is not $\cl_M$-rational.
\end{prop}

See \cite{LeuschkeWiegand} for a discussion of MCM modules over these 2-dimensional ADE singularities.

\begin{proof}
    Let $J=(x,y)$, then $z^2\in J$ implies that $\sqrt{J}=\mm$. Note that $z\nin J$ as $R/J\simeq k\ps{z}/(z^2)\neq k$. We will show that $z\in J^{\cl}$ where $\cl =\cl_M$.

    By definition, we need to show that for any $u\in M$, $zu\in JM$. Write $M$ as the cokernel of $R^n \xrightarrow{zI_n-\varphi} R^n$, let $\set{e_i}_{1\leqslant i\leqslant n}$ be a basis for the second copy of $R^n$, and let $\varphi_i$ denote the $i$th column of $\varphi$. Then $ze_i-\varphi_i=0$ in $M$, which implies that $ze_i\in JM$. Hence $zu \in JM$ for any $u\in M$.
\end{proof}

\begin{rmk}
All of the examples in \cite[Section 6]{benali2021cohen} satisfy the above theorem.
\end{rmk}

\section*{Acknowledgments}

We thank Neil Epstein for some very helpful comments on this paper.

\bibliography{main}
\bibliographystyle{halpha}
\end{document}